\documentclass{article}

\usepackage{amsmath, mathrsfs, amssymb, stmaryrd, cancel,tikz,amsthm,placeins}
\usepackage{graphicx}
\usepackage{xfrac}
\usepackage[all]{xy}
\usepackage[normalem]{ulem}
\usepackage{tikz-cd}
\usepackage{multirow}

\theoremstyle{plain}
\newtheorem{theorem}{Theorem}[section]{\bfseries}{\itshape}
\newtheorem{proposition}[theorem]{Proposition}{\bfseries}{\itshape}
\newtheorem{definition}[theorem]{Definition}{\bfseries}{\upshape}
\newtheorem{lemma}[theorem]{Lemma}{\bfseries}{\upshape}
{\bfseries}{\upshape}
\newtheorem{corollary}[theorem]{Corollary}{\bfseries}{\upshape}
{\bfseries}{\upshape}
{\bfseries}{\upshape}
{\bfseries}{\upshape}

\newcommand{\bw}{\bigwedge}
\newcommand{\bv}{\bigvee}
\newcommand{\UPP}{\prod_U P}
\newcommand{\g}{\bar{G}_f}

\newcommand{\B}{\textbf{Box}}

\begin{document}
\title{Non-elementary classes of representable posets}
\author{Rob Egrot}
\date{}

\maketitle

\begin{abstract}
A poset is $(\omega,C)$-representable if it can be embedded into a field of sets in such a way that all existing joins, and all existing \emph{finite} meets are preserved. We show that the class of $(\omega,C)$-representable posets cannot be axiomatized in first order logic using the standard language of posets. We generalize this result to $(\alpha,\beta)$-representable posets for certain values of $\alpha$ and $\beta$.    
\end{abstract}

\begin{section}{Introduction}
It is a trivial consequence of Stone's theorem \cite{Stone37} that every Boolean algebra is isomorphic to a field of sets. That is, every Boolean algebra is isomorphic to a subalgebra of a powerset equipped with set theoretic union, intersection and complementation. A generalization of this result can be obtained using the Prime Ideal Theorem for distributive lattices (attributed to Birkhoff). According to this generalization, a lattice is isomorphic to a ring of sets if and only if it is distributive. 

We will use the term \emph{representable} to describe various ordered structures that can be embedded into fields of sets by an embedding that preserves existing finite joins, meets and complements as set theoretic unions, intersections and complements. Note that preservation of existing binary meets and joins guarantees preservation of existing relative complements. For example, using this terminology we would say that every Boolean algebra is representable, and a lattice is representable if and only if it is distributive.  

Further generalization to the case of semilattices has also been made \cite{Bal69,Sch72}. A semilattice is representable if and only if it satisfies an infinite family of axioms generalizing distributivity. The question of whether this infinite family is equivalent to a finite subset has an interesting history. Schein \cite{Sch72} claimed that it is not, but did not provide a proof. Some years later this issue appeared in the literature as open \cite{PawTha80}. It was then shown that a semilattice satisfying only a finite (non-zero) number of the representation axioms must necessarily be infinite, and it was suggested that such a semilattice does not exist \cite{SCLS85}. A suitable infinite semilattice was constructed in \cite{Kea97}, finally settling the question in the negative.  

The class of representable posets is known to be elementary \cite{Egr16}, but direct generalizations of the axioms for the semilattice case fail, and explicit axioms are not known. It is known that the class of posets with representations preserving all or finitely bounded existing finite meets and joins cannot be finitely axiomatized (the result for finitely bounded meets and joins is proved in \cite{Egrsub2}, and the result for all finite meets and joins follows from the semilattice case \cite{Kea97}).      

We can also consider representations where \emph{infinite} meets and/or joins are preserved. A positive result is that a Boolean algebra has a representation preserving arbitrary meets and/or joins if and only if it is atomic \cite[corollary 1]{Abi71}, in which case we say it is \emph{completely representable}. The case of distributive lattices is studied in \cite{EgrHir12}, and that of posets in \cite{Egr16}. Unlike the Boolean case, a lattice representation may preserve all existing meets, for example, but not all existing joins. 

Unfortunately, in the lattice and poset cases there is no simple correspondent to the result for Boolean algebras. In particular, the class of distributive lattices with representations preserving all meets and joins is not closed under elementary equivalence \cite[theorem 3.2]{EgrHir12}, so neither the class of completely representable lattices nor the class of completely representable posets can be elementary. However, both these classes (and various others) can be shown to be pseudoelementary \cite[theorem 3.5]{EgrHir12}, \cite[theorem 5.7]{Egr16}. 

If $\alpha>2$ is a cardinal we say a poset is $(\alpha, C)$-representable if it has a representation preserving all existing meets of cardinality strictly less than $\alpha$ and all existing joins. It is conjectured in \cite{EgrHir12} that the class of $(\omega,C)$-representable lattices is not elementary, and it is conjectured in \cite{Egr16} that the class of $(\alpha, C)$-representable posets is not elementary for all choices of $\alpha$.

Our main result here proves the second conjecture and can be considered a step towards proving the first. We construct a poset that is not $(\omega,C)$-representable, but has an ultrapower that is. This implies that the class of $(\omega,C)$-representable posets is not elementary by \L o\'s' theorem. This construction also covers some other cardinalities relevant to our conjecture, as we see in the final section. These results deal with the cases where the question of the elementarity of the poset representation class is relatively difficult and, taken with previous results, settle the basic question of elementarity for all cardinalities. There remain some unanswered questions, regarding pseudoelementarity, for example, and we provide a table in the final section summarizing the current state of knowledge for easy reference.  

In section \ref{S:basic} we introduce some terminology. The rest of the document is concerned with the construction of a particular poset $P$, and proofs of various technical properties it possesses culminating in a proof of our main result. We assume a working knowledge of ultraproducts and \L o\'s' theorem (see e.g. \cite{ChaKei90} for a textbook treatment of these topics).
 
\end{section}

\begin{section}{Poset representations}\label{S:basic}
If $P$ is a poset, $S\subseteq P$, and $p\in P$ we use the following notational conventions:
\begin{itemize}
\item[$S^\uparrow$:] $=\{q\in P : q\geq s$ for some $s\in S\}$
\item[$p^\uparrow$:] $=\{p\}^\uparrow$
\end{itemize}
Following the notation of \cite{Egr16} we make the following definitions.

\begin{definition}[$(\omega,C)$-morphism]\label{D:mapdefs}
Given posets $P_1$ and $P_2$ we say a map $f\colon P_1\to P_2$ is an $(\omega,C)$-\emph{morphism} if $f(\bw S)=\bw f[S]$ whenever $\bw S$ is defined and $|S|<\omega$, and $f(\bv T)=\bv f[T]$ whenever $\bv T$ is defined. If $f$ is also injective we say it is an $(\omega,C)$-embedding (note that $f$ will always be order preserving).
\end{definition}

\begin{definition}[$(\omega,C)$-representation]\label{D:rep}
An $(\omega,C)$-\emph{representation} of a poset $P$ is an $(\omega,C)$-embedding $h\colon P \to \wp(X)$ for some set $X$ where $\wp(X)$ is considered as a field of sets. When $P$ has a top and/or bottom, we demand that $h$ maps them to $X$ and/or $\emptyset$ respectively.
\end{definition}

\begin{definition}[$(C,\omega)$-ideal]\label{D:wfilter}
$\emptyset\neq S\subset P$ is a $(C,\omega)$-\emph{ideal} of $P$ if it is closed downwards and for all $X\subseteq S$ we have $\bv X\in S$ whenever $\bv X$ is defined, and whenever $Y\subseteq P$ with $|Y|<\omega$ and $\bw Y$ defined in $P$ we have $\bw Y\in S\implies y\in S$ for some $y\in Y$.
\end{definition}  

\begin{definition}[$\omega$-ideal]
$\emptyset\neq S\subset P$ is an $\omega$-\emph{ideal} of $P$ if it is closed downwards, closed under existing finite joins, and whenever $Y$ is a finite subset of $P$ with $\bw Y$ defined and in $S$ there is $y \in Y\cap S$.
\end{definition}

\begin{definition}[Down-separating]
$S\subseteq \wp(P)$ is \emph{down-separating} over $P$ if whenever $p\not\leq q$ there is $X\in S$ with $q\in X$ and $p\notin X$.
\end{definition}

We can use down-separation to characterize the $(\omega,C)$-representable posets. The following is a special case of \cite[theorem 2.7]{Egr16}, and the concept of using separation by sets of ideals and filters to characterize representability is referenced as far back as \cite{ChaHor62}.
\begin{theorem}\label{T:rep}
For a poset $P$ the following are equivalent:
\begin{enumerate}
\item $P$ is $(\omega,C)$-representable,
\item The set of $(C,\omega)$-ideals of $P$ is down-separating over $P$.
\end{enumerate}
\end{theorem}

Note that in the above theorem $(\omega,C)$-representable corresponds to separation by $(C,\omega)$-ideals (reversing the order of $C$ and $\omega$). This notational irregularity is an unfortunate artifact of the duality between filters and ideals in the system used in \cite{Egr16}, and the convenience of using ideals over filters in our construction here. 
\end{section}

\begin{section}{Building $P$}\label{S:build}
We intend to use the fact that an elementary class must be closed under ultraroots. We will construct a countable poset $P$ that is not $(\omega,C)$-representable but that has an ultrapower that is $(\omega,C)$-representable. This section is devoted to the construction of $P$. We prove the relevant claims in the following sections.

\begin{lemma}\label{L:X}
Let $\omega+1=\{0,1,2,\dots,\omega\}$. Then for each $n\in\omega+1$ we can define a countable set $X_n\subset(0,1]\subset \mathbb{R}$ such that the following properties hold:
\begin{enumerate}
\item[P1:] $X_n$ is a dense subset of $(0,1]$ for all $n\leq \omega$.
\item[P2:] $X_m\cap X_n=\{1\}$ for all $m,n<\omega$.
\item[P3:] $X_n\cap X_\omega=\emptyset$ for all $n< \omega$.
\end{enumerate}
\begin{proof}
Let $\{p_0,p_1,p_2,\ldots\}$ be an enumeration of the primes. We can, for example, take $X_\omega=\mathbb{Q}\cap(0,1)$, and for each $n\in\omega$ define
\begin{equation*}
X_n=\{r\in(0,1):r=q.\sqrt{p_n} \text{ for some }q\in \mathbb{Q}\}\cup\{1\}
\end{equation*} 
\end{proof}
\end{lemma}

To construct $P$ we choose any $X_0,X_1,\ldots,X_\omega$ satisfying the conditions of lemma \ref{L:X}. We proceed by making the following definitions:

\begin{itemize}
\item $(0,1]$ is taken as a subset of $\mathbb{R}$.
\item $\mathbb{Q}^+=[0,\infty)\cap\mathbb{Q}$, i.e. $\mathbb{Q}^+$ is the set of non-negative rationals.
\item Given $q\in \mathbb{Q}^+$ we write $\lfloor q \rfloor$ for the largest integer smaller than (or equal to) $q$.
\item $\pi_x\colon (0,1]\times \mathbb{Q}^+\to (0,1]$ is defined by $\pi_x(x,y)=x$.  
\item $\pi_y\colon (0,1]\times \mathbb{Q}^+\to \mathbb{Q}^+$ is defined by $\pi_y(x,y)=y$. I.e. $\pi_x$ and $\pi_y$ are projection functions.
\item $P'$ is the subset of $(0,1]\times \mathbb{Q}^+$ composed of all pairs $(x,y)$ such that $x\in X_{\lfloor y \rfloor}$. We define an ordering on $P'$ by 
\begin{equation*}
a\leq b \iff \pi_x(a)\leq \pi_x(b) \text{ and } \pi_y(a)\leq \pi_y(b)
\end{equation*}
I.e. $P'$ inherits the product ordering on $(0,1]\times \mathbb{Q}^+$. 
\item $P$ is $P'$ capped by a top element $\top$.
\item Abusing notation a little we define $\pi_x\colon P\to (0,1]$ by setting 
\begin{equation*}
\pi_x(a) = \begin{cases} \pi_x(a) \text{ when } a\in (0,1]\times \mathbb{Q}^+ \\ \pi_x(\top)=1\end{cases} 
\end{equation*} 
\item Similarly we define $\pi_y\colon P\to \mathbb{Q}\cup\{\omega\}$ by setting 
\begin{equation*}
\pi_y(a) = \begin{cases} \pi_y(a) \text{ when } a\in (0,1]\times \mathbb{Q}^+ \\ \pi_y(\top)=\omega\end{cases} 
\end{equation*} 
\item $\B\colon P\to\omega+1$ is defined by \begin{equation*}\B(p)=\begin{cases}\lfloor\pi_y(p) \rfloor \text{ when $p\in P'$} \\ \omega \text{ when $p=\top$}\end{cases}\end{equation*}
\end{itemize}
Unpacking the definition of $P$ we see that it is defined by taking $\bigcup_{n\in\omega} \big(X_n\times ([n,n+1)\cap \mathbb{Q})\big)$, viewing it as a subposet of $(0,1]\times \mathbb{Q}^+$ ordered using the product ordering, then affixing a top element. Figure \ref{F:P} has a diagram. Here $p$ is a point of $P$, the shaded area to the upper right of $p$ marks the set of elements (other than $\top$) greater than $p$ in $P$, while the shaded area to the bottom left of $p$ marks those elements of $P$ that are less than $p$ in $P$. We think of $P$ is being composed of `boxes' stacked on top of each other (with a top added). The $\B$ function returns the number of the box containing the given element. In this case $\B(p)= 1$. 

\begin{figure}
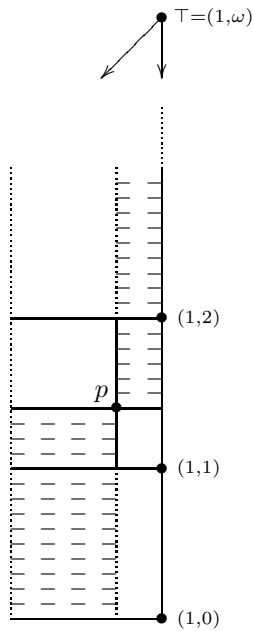
\[
\xy<1cm,0cm>:
(0,8);(0,6)**@{.},
(2,8);(2,6)*={\bullet}**@{-},
(0,6);(2,6)**@{-},
(2.5,6)*{_{(1,2)}},
(0,6);(0,4)**@{.},
(2,6);(2,4)*={\bullet}**@{-},
(0,4);(2,4)**@{-},
(2.5,4)*{_{(1,1)}},
(2.5,2)*{_{(1,0)}},
(0,4);(0,2)**@{.},
(0,2);(2,2)**@{-},
(2,4);(2,2)*={\bullet}**@{-},
(1.2,5)*={p},
(1.4,4.8)*={\bullet},
(2,8);(2,8.8)**@{.},
(2.7,10)*{_{\top= (1,\omega)}},
(2,10)*={\bullet};(2,9.2)**@{-}?>* \dir{>},
(2,10);(1.2,9.2)**@{-}?>* \dir{>},
(1.4,4);(1.4,6)**@{-},
(1.4,5);(2,5)**@{--},
(1.4,5.2);(2,5.2)**@{--},
(1.4,5.4);(2,5.4)**@{--},
(1.4,5.6);(2,5.6)**@{--},
(1.4,5.8);(2,5.8)**@{--},
%(1.4,6);(2,6)**@{--},
(1.4,6.2);(2,6.2)**@{--},
(1.4,6.4);(2,6.4)**@{--},
(1.4,6.6);(2,6.6)**@{--},
(1.4,6.8);(2,6.8)**@{--},
(1.4,7);(2,7)**@{--},
(1.4,7.2);(2,7.2)**@{--},
(1.4,7.4);(2,7.4)**@{--},
(1.4,7.6);(2,7.6)**@{--},
(1.4,7.8);(2,7.8)**@{--},
(0,4.8);(2,4.8)**@{-},
(0,4.6);(1.39,4.6)**@{--},
(0,4.4);(1.39,4.4)**@{--},
(0,4.2);(1.39,4.2)**@{--},
(0,3.8);(1.39,3.8)**@{--},
(0,3.6);(1.39,3.6)**@{--},
(0,3.4);(1.39,3.4)**@{--},
(0,3.2);(1.39,3.2)**@{--},
(0,3);(1.39,3)**@{--},
(0,2.8);(1.39,2.8)**@{--},
(0,2.6);(1.39,2.6)**@{--},
(0,2.4);(1.39,2.4)**@{--},
(0,2.2);(1.39,2.2)**@{--},
(1.4,6);(1.4,8)**@{.},
(1.4,4);(1.4,2)**@{.},
\endxy\]
\caption{\label{F:P}The poset $P$}
\end{figure}
\end{section}

\begin{section}{Some properties of $P$}

\begin{lemma}\label{L:xdiff}
If $p,q\in P$ with $\B(p)\neq\B(q)$. Then $\pi_x(p)=\pi_x(q)\iff \pi_x(p)=\pi_x(q)=1$.
\end{lemma}
\begin{proof}
This is a trivial consequence of lemma \ref{L:X} (P2).
\end{proof}

\begin{lemma}\label{L:inBetween}
If $p<q\in P$ with $\B(p)< \B(q)$ we can find $r\in P$ with $\pi_y(p)=\pi_y(r)$ and $\pi_x(p)<\pi_x(r)<\pi_x(q)$.
\end{lemma}
\begin{proof}This is a trivial consequence of lemma \ref{L:X} (P1). 
\end{proof}

\begin{lemma}\label{L:meet}
Let $p$ and $q\in P$ with $\B(p)=\B(q)=n$. Then \begin{equation*}p\wedge q=\big(\min(\pi_x(p),\pi_x(q)),\min(\pi_y(p),\pi_y(q))\big)\end{equation*}
\end{lemma}
\begin{proof}
This follows immediately from the definition of the order on $P$.
\end{proof}

\begin{lemma}\label{L:onlyMeet}
Let $S\subset P$ be finite. Then $\bw S$ is defined if and only if there is $n\in\omega$ such that, for all $s\in S$, either $\B(s)=n$ or $s\geq s'$ for some $s'\in S$ with $\B(s')=n$.
\end{lemma}
\begin{proof}
The `if' part follows from an easy generalization of lemma \ref{L:meet}. For `only if' we proceed by proving a sequence of sub-claims culminating in a proof of the main statement. 
\begin{enumerate}
\item ``If $\bw S$ exists then we must have $\pi_x(\bw S)\leq \min(\{\pi_x(s):s\in S\})"$\\
This is automatic from the definition of the order on $P$.
\item``If $\bw S$ exists then we must have $\pi_y(\bw S)= \min(\{\pi_y(s):s\in S\})$"\\
As above we must have $\pi_y(\bw S)\leq \min(\{\pi_y(s):s\in S\})$. If $\pi
_y(\bw S)<\min(\{\pi_y(s):s\in S\})$ then we obtain a contradiction by finding a new candidate for the infimum $(\pi_x(\bw S),\pi_y(\bw S)+\epsilon)$ where $\epsilon$ is some suitably small rational value of our choice.
\item ``If $s\in S$ with $\B(s)\neq\B(\bw S)=n$ and $s\not\geq s'$ for all $s'\in S$ with $\B(s')=n$ then $\pi_x(s)<\pi_x(s')$ for all $s'\in S$ with $\B(s')=n$"\\
We must have $\B(s)>n$ so $\pi_y(s)>\pi_y(s')$ for all $s'\in S$ with $\B(s')=n$. Since for all such $s'$ we have $s\not\geq s'$ we must have $\pi_x(s)<\pi_x(s')$.
\item ``The lemma is true" \\
We prove this by contradiction. Suppose $\bw S$ is defined but there is non-empty $T\subset S$ with $\B(t)>\B(\bw S)=n$ and $t\not\geq s'$ for all $t\in T$ and for all $s'\in S$ with $\B(s')=n$. Since $S$ is finite we can choose $t'\in T$ so that $\pi_x(t')=\min(\{\pi_x(t):t\in T\})$. By (3) we have $\pi_x(t')\leq\pi_x(s)$ for all $s\in S$, and we must have $\pi_x(\bw S)\leq \pi_x(t')$ by (1). Since $\B(t')>\B(\bw S)$ we must have $\pi_x(\bw S)< \pi_x(t')$ (by lemma \ref{L:xdiff}, as we cannot have $\pi_x(t')=1$). By lemma \ref{L:inBetween} and (3) we can find $r\in P$ with $\bw S< r < s$ for all $s\in S$, but this is a contradiction on the definition of $\bw S$.   
\end{enumerate}
\end{proof}

\begin{corollary}\label{C:meets}
Let $S\subset P$ be finite. The $\bw S$ is defined in $P$ if and only if there is $T\subseteq S$ and $n\in\omega$ with $\B(t)=n$ for all $t\in T$, and for each $s\in S$ there is $t\in T$ with $s\geq t$. In this case $\bw T = \bw S = (\min(\{\pi_x(s):s\in S\}),\min(\{\pi_y(s):s\in S\}))$.
\end{corollary}
\begin{proof}
The if and only if statement follows directly from lemma \ref{L:onlyMeet}, and the rest follows from an easy generalization of lemma \ref{L:meet}.
\end{proof}

Analogous versions of lemmas \ref{L:meet} and \ref{L:onlyMeet}, and corollary \ref{C:meets} hold for joins. 

\begin{lemma}\label{L:Join}
Let $q\in P$ and let $\pi_y(q)=n+1$. Then $q=\bv\{p\in P : \B(p)=n$ and $\pi_x(p)<\pi_x(q)\}$.
\end{lemma}
\begin{proof}
Let $S=\{p\in P : \B(p)=n$ and $\pi_x(p)<\pi_x(q)\}$. If $r\in P$ and $r\geq p$ for all $p\in S$ then $\pi_y(r)\geq \pi_y(q)$ by density of $\mathbb{Q}$, and $\pi_x(r)\geq \pi_x(q)$ by lemma \ref{L:X} (P1). Since $q=(\pi_x(q),\pi_y(q))$ we are done.
\end{proof}

\begin{definition}[meet-prime]
If $Q$ is any poset then $q\in Q$ is \emph{meet-prime} if, for all finite $S\subseteq Q$ with $\bw S$ defined in $P$, whenever $\bw S\leq q$ there is $s\in S$ with $s\leq q$.
\end{definition}

\begin{lemma}\label{L:prime}
Let $p=(1,y)\in P'$. Then $p$ is meet-prime in $P$.
\end{lemma}
\begin{proof}
Let $S\subset P$ be finite. By lemma \ref{L:meet} we know that $\bw S = (\min(\{\pi_x(s):s\in S\}),\min(\{\pi_y(s):s\in S\})$ whenever it is defined, so if $\bw S \leq p$ then there is $s\in S$ with $\pi_y(s)\leq\pi_y(p)$. Since $\pi_x(p)=1$ we must also have $\pi_x(s)\leq\pi_x(p)$ and thus $s\leq p$ as required.
\end{proof}

\begin{lemma}\label{L:verticalIdeal}
Let $r\in X_\omega$. Then $\gamma=\{p\in P: \pi_x(p)<r\}$ is an $\omega$-ideal of $P$.
\end{lemma}
\begin{proof}
Downward closure is automatic. Let $S$ be a finite subset of $\gamma$ such that $\bv S$ is defined in $P$. Then by the `join' version of corollary \ref{C:meets} we have $\pi_x(\bv S) =\max(\{\pi_x(s):\in S\})$, and $\pi_x(s)<r$ for all $s\in S$, so $\pi_x(\bv S)< r$ and $\bv S \in \gamma$ as required. Finally, if $S$ is a finite subset of $P$ with $\bw S$ defined then by corollary \ref{C:meets} we must have $\pi_x(\bw S)=\min(\{\pi_x(s):s\in S\})$. So if $\pi_x(\bw S)< r$ then there is $s\in S$ with $\pi_x(s)< r$ too. 
\end{proof}

\end{section}

\begin{section}{$P$ is not $(\omega,C)$-representable}
\begin{lemma}\label{L:notRep}
Let $p,q\in P$ be such that $\pi_y(p)=\pi_y(q)$ and $\pi_x(p)<\pi_x(q)$. Then there is no $(C,\omega)$-ideal containing $p$ but not $q$.
\end{lemma}
\begin{proof}
Suppose $\B(p)=\B(q)=n$. Let $\gamma$ be a $(C,\omega)$-ideal containing $p$. We will show that $\gamma$ contains $q$. We proceed by proving sub-claims as follows:
\begin{enumerate}
\item ``Either $\gamma$ contains $q$ or $\gamma$ contains some $q'$ such that $\B(q')=n+1$"\\
Pick any $q'\in P$ so that $\pi_y(q')=n+1$ and $\pi_x(q')<\pi_x(p)$. Let $S=\{r\in P:\B(r)=n$, $\pi_y(r)\geq \pi_y(p)$ and $\pi_x(r)<\pi_x(q')\}$. Then $q'=\bv S$ by lemma \ref{L:Join}. Let $r\in S$. Then $r\wedge q\leq p$ by lemma \ref{L:meet}, so as $\gamma$ is a $(C,\omega)$-ideal we must have either $r\in\gamma$ or $q\in\gamma$. So if $q\notin \gamma$ then $S\subseteq \gamma$ and thus $q'\in\gamma$ as $q'=\bv S$.  So either $q\in\gamma$ or $q'\in\gamma$ as claimed.

\item ``For all $k\in\omega$, either $\gamma$ contains $q$ or $\gamma$ contains some $q_k$ such that $\B(q_k)=k$"\\
This follows fairly easily by induction using part (1). Having used the argument in the $k$th box to obtain $q_{k+1}\in\gamma$ with $\pi_y(q_{k+1})= k+1$, we repeat it using $q_{k+1}$ in place of $p$ and $q^*$ in place of $q$ for some $q^*$ with $q<q^*$ and $\pi_y(q^*)=k+1$ (we could use $q^*=(1,k+1)$ for example).

\item ``$\gamma$ contains $q$" \\
Since for all $k\in\omega$ there is $q_k\in\gamma$ such that $\B(q_k)=k$ we note that $\bv_{k\in\omega} q_k =\top$, and thus that $\top\in\gamma$. Since $q\leq\top$ we must conclude that $q\in\gamma$.
\end{enumerate} 
\end{proof}
\end{section}

\begin{section}{An ultrapower of $P$ that is $(\omega,C)$-representable}
Let $U\subset \wp(\omega)$ be a non-principal ultrafilter. Let $\UPP$ be the ultrapower of $P$ over $U$. We will show that $\UPP$ is $(\omega,C)$-representable by constructing separating $(C,\omega)$-ideals and appealing to theorem \ref{T:rep}. The main result will then be the following.

\begin{theorem}\label{T:main}
The class of $(\omega,C)$-representable posets is not elementary. 
\end{theorem}
\begin{proof}
The poset $P$ constructed in section \ref{S:build} is not $(\omega,C)$-representable by lemma \ref{L:notRep} and theorem \ref{T:rep}. By a straightforward argument using the definition of the ultrapower and elementary properties of ultrafilters, if $[a]\not\leq [b]$ in $\UPP$ then either
\begin{enumerate}
\item $\{i\in\omega:\pi_y(b(i))<\pi_y(a(i))\}\in U$, and/or
\item $\{i\in\omega:\pi_x(b(i))<\pi_x(a(i))\}\in U$.
\end{enumerate}
There is a $(C,\omega)$-ideal of $\UPP$ containing $[b]$ but not $[a]$ by lemma \ref{L:vertical}, for case 1, and by corollary \ref{P:horizontal} for case 2. Thus $\UPP$ is $(\omega,C)$-representable by theorem \ref{T:rep}. Since an elementary class must be closed under ultraroots we are done.
\end{proof}

\begin{lemma}\label{L:vertical}
Let $[a]\not\leq[b]\in \UPP$ and suppose $\{i\in\omega:\pi_y(b(i))<\pi_y(a(i))\}=u\in U$. Then there is a $(C,\omega)$-ideal of $\UPP$ containing $[b]$ but not $[a]$.
\end{lemma}
\begin{proof}
For each $i\in u$ pick $y_i\in\mathbb{Q}^+$ so that $\pi_y(b(i))\leq y_i<\pi_y(a(i))$. Define $[c]\in\UPP$ by $c(i)=(1,y_i)$ for all $i\in u$. Then $c(i)$ is meet-prime for all $i\in u$ by lemma \ref{L:prime}, and so $[c]$ is meet-prime by \L o\'s' theorem. So $[c]^\downarrow$ is a $(C,\omega)$-ideal containing $[b]$ but not $[a]$ as required.
\end{proof}

Lemma \ref{L:vertical} shows how we can find $(C,\omega)$-ideals separating elements with different $\pi_y$ values. To complete the proof of theorem \ref{T:main} it remains to construct $(C,\omega)$-ideals separating elements with different $\pi_x$ values. This is more difficult, as our ideals must avoid containing sequences of ultrapower elements whose supremum is $\top$. The solution to this problem is provided by definition \ref{D:G} and proposition \ref{P:main}. 

\begin{lemma}\label{L:Delta}
If $S=\{[s_1],\ldots,[s_n]\}$ is a finite subset of $\UPP$ such that $\bw S$ exists then there is $T\subseteq S$ such that $\bw T = \bw S$ and, if $T =\{[t_1],\ldots,[t_k]\}$, we have $\{i\in\omega:\B(t_1(i))=\B(t_2(i))=\ldots=\B(t_k(i))\}\in U$
\end{lemma}
\begin{proof}
The strategy is to extend the signature for the language of posets and apply \L o\'s'  theorem. 
\begin{itemize}
\item For each $n\in\omega$ define a $2n$-place predicate $\textbf{S}_n$. The interpretation of this predicate holds in $P$ if and only if the set defined by the assignment of the first $n$ variables is a subset of the set defined by the last $n$ variables. I.e. \begin{equation*}
P\models \textbf{S}_n(x_1,\ldots,x_n,y_1,\ldots,y_n)\iff \{x_1,\ldots,x_n\}\subseteq \{y_1,\ldots,y_n\}\end{equation*}
\item For each $n\in\omega$ define an $n+1$-place predicate $\textbf{glb}_n$. The interpretation of this predicate holds in $P$ if the element defined by the first variable is the meet in $P$ of the set defined by the last $n$ variables. I.e.  \begin{equation*}P\models \textbf{glb}_n(x,y_1,\ldots,y_n) \iff x = \bw \{y_1,\ldots,y_n\}   \end{equation*} 
\item For each $n\in \omega$ define an $n$-place predicate $\textbf{B}_n$. The interpretation of this predicate holds in $P$ if the $\B$ values of every element are equal. I.e. \begin{equation*}P\models \textbf{B}_n(x_1,\ldots,x_n)\iff \B(x_1)=\ldots=\B(x_n) \end{equation*}
\end{itemize}
Note that $\textbf{S}_n$ and $\textbf{glb}_n$ are definable in the standard language of posets. Having defined these predicates, for each $n\in\omega$ we define the following sentence in the extended language.
\begin{align*}
\phi_n=\forall x_1\ldots x_n\forall y\Big(\textbf{glb}_n(y,x_1,\ldots,x_n)\rightarrow \exists z_1\ldots z_n\big(&\textbf{S}_n(z_1,\ldots,z_n,x_1,\ldots,x_n) \\
\wedge& \textbf{B}_n(z_1,\ldots,z_n)\\
\wedge &\textbf{glb}_n(y,z_1,\ldots,z_n)\big)\Big)
\end{align*}
We know that $\phi_n$ holds in $P$ for all $n$ by corollary \ref{C:meets}, so by \L o\'s'  theorem they all hold in $\UPP$. The interpretation of the $\phi_n$ in $\UPP$, taken in sum, is precisely the statement we are trying to prove.  
\end{proof}

\begin{definition}[$\bar{S}$]\label{D:barS} Let $Q$ be any poset, let $I$ be an indexing set, let $U$ be any non-principal ultrafilter of $\wp(I)$ and let $u\in U$. Let $S_i\subseteq Q$ for each $i\in u$ and let $S=(S_i:i\in u)$. Then we define $\bar{S}=\{[a]\in\prod_U Q:\{i\in u:a(i)\in S_i\}\in U\}$. 
\end{definition}
For an alternative perspective on definition \ref{D:barS} we can think of each set $S_i$ as the interpretation of a unary predicate in $Q$. Then $\bar{S}$ is the interpretation of that predicate in $\prod_U Q$.

\begin{lemma}\label{L:ideal}
With notation as in definition \ref{D:barS}, if $S_i$ is an $\omega$-ideal of $Q$ for all $i\in u$ then $\bar{S}$ is an $\omega$-ideal of $\prod_U Q$.
\end{lemma}
\begin{proof}
This follows easily from the characterization of each $S_i$ as the interpretation of a unary predicate and \L o\'s' theorem.
\end{proof}

\begin{definition}[$\g$]\label{D:G}
Returning to our poset $P$ and the ultrapower $\UPP$, let $u\in U$ and for each $i\in u$ let $r_i\in X_\omega$ and let $\gamma_i$ be an $\omega$-ideal of $P$ of the form $\gamma_i=\{p\in P:\pi_x(p)<r_i\}$. Let $G=(\gamma_i:i\in u)$, and let $f\colon\omega\to\omega$ be a function. Then we define 
\begin{equation*}
\g=\{[a]\in\bar{G}:\exists z\in\mathbb{Z}(\{i\in \omega:\textnormal{\textbf{Box}}(a(i))\leq f(i)+z\}\in U)\}
\end{equation*}
\end{definition}

\begin{proposition}\label{P:main}
Let $f:\omega\to\omega$ and let $\Gamma =\g$. Then $\Gamma$ is a $(C,\omega)$-ideal of $\UPP$.
\end{proposition}
\begin{proof}
A straightforward argument using properties of ultrafilters and lemma \ref{L:ideal} verifies that $\Gamma$ is closed downwards, so we complete the proof by proving two statements.
\begin{enumerate}
\item ``If $S=\{[s_1],\ldots,[s_n]\}$ is a finite subset of $\UPP$ such that $\bw S$ exists and is in $\Gamma$ then there is $[s]\in S\cap \Gamma$"\\
By lemma \ref{L:Delta} there is $T=\{[t_1],\ldots,[t_k]\}\subseteq S$ with $\bw T=\bw S$ and $\{i\in\omega:\B(t_1(i))=\ldots=\B(t_k(i))\}\in U$. We proceed by proving sub-claims as follows:
\begin{enumerate}
\item ``$[t]\in \bar{G}$ for some $[t]\in T$"\\
We have $\bw T=\bw S \in \bar{G}$. The statement follows by lemma \ref{L:ideal}.
\item ``For the $[t]$ from part (a) we have $[t]\in \Gamma\cap S$"\\
By choice of $T$ and $t$ we have $\{i\in \omega: t(i)\in\gamma_i$,  $\bw \{t_1(i),\ldots,t_k(i)\} = \bw \{s_1(i),\ldots, s_n(i)\}$ and $\B(t_1(i))=\ldots=\B(t_k(i))\}=u_0$ for some $u_0\in U$. Since $\bw T\in\Gamma$ there must be $z\in \mathbb{Z}$ such that $\{i\in u_0: \B(\bw\{t_1(i),\ldots,t_n(i)\})\leq f(i)+z\}=u_1\in U$. Note that for all $i\in u_0$ we have $\B(\bw\{t_1(i),\ldots,t_n(i)\})=\B(t(i))$, and so $\B(t(i))\leq f(i)+ z$ for all $i\in u_1$, and thus $[t]\in \Gamma$ as required. 
\end{enumerate} 
\item ``Let $S\subseteq \Gamma$ and let $[b]$ be an upper bound for $S$. Then either $[b]\in\Gamma$ or there is an upper bound $[b']$ for $S$ with $[b]\not\leq[b']$"\\
We have three cases. In each case we either show that $[b]\in\Gamma$ or construct a suitable $[b']$.
\begin{itemize}
\item Case 1: \emph{$[b]=[\top]$} \\
We construct $[b']$. Here we define $b'$ by $b'(i)=(1,f(i) + i)$. Then $[b']<[b]$. Let $[a]\in\Gamma$. Then $\{i\in\omega:\pi_x(a(i))<r_i\}\in U$, so $\{i\in\omega:\pi_x(a(i))<\pi_x(b'(i))\}\in U$. Also, by definition of $\Gamma$ there is $z\in \mathbb{Z}$ with $\{i\in\omega:\B(a(i))\leq f(i)+z\}\in U$. Since $\{i\in\omega: i> z\}\in U$ we must have $\{i\in \omega : b'(i)>a(i)\}\in U$. So $[b']$ is an upper bound for $\Gamma$, and thus also for $S$ as required.
\item Case 2: \emph{$[b]\neq[\top]$ and $[b]\not\in \bar{G}$}\\
We construct $[b']$. By assumption $\{i\in\omega:\pi_x(b(i))> r_i$ and $b(i)\neq\top\}=u_0$ for some $u_0\in U$. For each $i\in u_0$ we can find $x_i\in X_{\B(b(i))}$ with $r_i<x_i<\pi_x(b(i))$. Define $b'$ by $b'(i)=(x_i,\pi_y(b(i)))$ for $i\in u_0$. Clearly $[b']<[b]$. Moreover, by properties of ultrafilters, if $[s]\in S$ then $\{i\in \omega: \pi_x(b'(i))> \pi_x(s(i))$ and $\pi_y(b'(i))\geq \pi_y(s(i))\}\in U$. So $[b']$ is an upper bound for $S$ as required.
\item Case 3:  $[b]\in \bar{G}$ \\
Suppose $\{i\in\omega : b(i)\in \gamma_i\}=u_0\in U$. If there is $z\in\mathbb{Z}$ such that $\{i\in\omega : \B(b(i))\leq f(i)+z\}\in U$ then $[b]\in\Gamma$ and we are done. So suppose that $\{i\in u_0: \B(b(i))> f(i)+z\}\in U$ for all $z\in \mathbb{Z}$. We construct $[b']$. Let $\{i\in\omega: \B(b(i))> f(i)$ and $b(i)\neq \top\}=u_1\in U$ (we must have $[b]\neq[\top]$ as $[b]\in\bar{G}$). We define $[b']$ by $b'(i)=(1, \B(b(i))-1)$ for $i\in u_1$. Let $[a]\in \Gamma$. Then $\{i\in \omega:\pi_x(a(i))< r_i< 1\}\in U$, and $\{i\in\omega : \pi_y(a(i))\leq f(i)+ z\}\in U$ for some $z\in \mathbb{Z}$. But $\{i\in\omega : \B(b(i))-1>f(i)+z\}\in U$, so $[a]<[b']$ and thus $[b']$ is an upper bound for $\Gamma$ (and so also for $S$). Since $\pi_y(b'(i))<\pi_y(b(i))$ for all $i\in u_1$ we have $[b]\not\leq[b']$ and we are done. 
\end{itemize}
\end{enumerate}
We have now proved claims (1) and (2), thus proving that $\Gamma$ is an $(C,\omega)$-ideal of $\UPP$ as required.
\end{proof}

\begin{corollary}\label{P:horizontal}
Let $[a]\not\leq[b]\in \UPP$ and suppose $\{i\in\omega:\pi_x(b(i))<\pi_x(a(i))\}=u\in U$. Then there is a $(C,\omega)$-ideal of $\UPP$ containing $[b]$ but not $[a]$.
\end{corollary}
\begin{proof}
For each $i\in u$ we choose $r_i\in X_\omega$ so that $\pi_x(b(i))<r_i<\pi_x(a(i))$ and define $\gamma_i=\{p\in P:\pi_x(p)<r_i\}$. We let $G=(\gamma_i:i\in u)$, we define $f\colon \omega\to \omega$ by $f(i)=\B(b(i))$, and we generate $\Gamma=\g$ as in definition \ref{D:G}. Then $\Gamma$ is a $(C,\omega)$-ideal by proposition \ref{P:horizontal}, and $[b]\in\Gamma$ by definition. Moreover, $[a]\notin \Gamma$ so we are done.
\end{proof}

\end{section}

\begin{section}{$(\alpha,\beta)$-representations}

Given cardinals $\alpha$ and $\beta$ with $\alpha,\beta>2$ we can expand on this result by using the following variations of definitions \ref{D:mapdefs} and \ref{D:rep}.

\begin{definition}[$(\alpha,\beta)$-morphism]
Given posets $P_1$ and $P_2$ we say a map $f\colon P_1\to P_2$ is an $(\alpha,\beta)$-morphism if $f(\bw S)=\bw f[S]$ whenever $\bw S$ is defined and $|S|<\alpha$, and $f(\bv T)=\bv f[T]$ whenever $\bv T$ is defined and $|T|<\beta$. If $f$ is also injective we say it is an $(\alpha,\beta)$-embedding (note that $f$ will always be order preserving).
\end{definition}

\begin{definition}[$(\alpha,\beta)$-representation]
An $(\alpha,\beta)$-representation of a poset $P$ is an $(\alpha,\beta)$-embedding $h\colon P \to \wp(X)$ for some set $X$ where $\wp(X)$ is considered as a field of sets. When $P$ has a top and/or bottom, we demand that $h$ maps them to $X$ and/or $\emptyset$ respectively.
\end{definition}

For any cardinals $\alpha,\beta>2$ the class of $(\alpha,\beta)$-representable posets can be defined using separation by $(\beta,\alpha)$-ideals in essentially the same manner as in theorem \ref{T:rep} (details can be found in \cite{Egr16}). We note that lemma \ref{L:notRep} uses only binary meets, so actually demonstrates that $P$ is not $(\alpha,\beta)$-representable for all $2<\alpha$ and $\omega<\beta$. Furthermore, it is a trivial consequence of the proof of theorem \ref{T:main} that $\UPP$ is $(\alpha,\beta)$-representable for all $2<\alpha\leq \omega$ and all $\beta$. Consequently, after an appeal to order duality, we obtain the following generalization of theorem \ref{T:main}.

\begin{theorem}\label{T:main2}
Let $\alpha$ and $\beta$ be cardinals with $2<\alpha\leq\omega$ and $\beta>\omega$. Then neither the class of $(\alpha,\beta)$-representable posets, nor the class of $(\beta,\alpha)$-representable posets is elementary. Moreover, the classes of $(\alpha,C)$ and $(C,\alpha)$-representable posets also fail to be elementary. 
\end{theorem}

We know from \cite{Egr16} that when $2<\alpha,\beta\leq \omega$ the class of $(\alpha,\beta)$-representable posets is elementary, and also that the class of $(C,C)$-representable posets is not elementary. We can also generalize \cite[theorem 5.1]{Egr16} as follows.

\begin{theorem}
Let $\alpha,\beta\geq\omega_1$. Then none of the classes of $(\alpha,\beta)$, $(\alpha,C)$, $(C,\beta)$, or $(C,C)$-representable lattices is closed under elementary equivalence.
\end{theorem}
\begin{proof}
Consider $L=[0,1]\subset \mathbb{R}$ and $L'=[0,1]\subset \mathbb{Q}$ considered as distributive lattices under the usual ordering. Then $L$ and $L'$ are elementarily equivalent. Since $\mathbb{R}$ can be constructed from $\mathbb{Q}$ using Dedekind cuts there can be no non-trivial $(\omega_1,\omega_1)$-ideals in $L$, as the supremum of any down-set of $L$ is definable from both above and below by countable sequences. So $L$ cannot be in any of the stated representation classes. However, given $a<b\in L'$ we can take $r\in \mathbb{R}$ with $a<r<b$ and let $\gamma =r^\downarrow\cap L'$. Then $\gamma$ is a $(C,C)$-ideal containing $a$ but not $b$, and we conclude that $L'$ is $(C,C)$-representable (and thus in every stated representation class) by the generalized version of theorem \ref{T:rep} (\cite[theorem 2.7]{Egr16}). Thus the stated classes are not closed under elementary equivalence. 
\end{proof}
\begin{figure}[ht]
\centering

\begin{tabular}{ll|llll}
               & $\alpha$        & \multirow{2}{*}{$<\omega$} & \multirow{2}{*}{$=\omega$} & \multirow{2}{*}{$>\omega$} & \multirow{2}{*}{C}\\ 
$\beta$        &                 &                 &              &                            &                    \\ \hline
\multicolumn{2}{l|}{$<\omega$}   & $E$               & $E$             &  $N$                          &  $NP^*$                  \\
\multicolumn{2}{l|}{$=\omega$}   & $E$               & $E$             & $N$                         & $NP$                  \\
\multicolumn{2}{l|}{$>\omega$}   & $N$              & $N$             & $N$                          & $N$                  \\
\multicolumn{2}{l|}{C}           & $NP^*$               & $NP$             & $N$                          & $NP^*$                 
\end{tabular}
{\footnotesize
\begin{itemize}
\item $E$ = elementary but not finitely axiomatizable
\item $NP$ = pseudoelementary but not elementary
\item $NP^*$ = basic pseudoelementary but not elementary
\item $N$ = not elementary but maybe pseudoelementary
\end{itemize}
}
\caption{Axiomatizing representation classes for various values of $\alpha$ and $\beta$
}
\label{F:elem}
\end{figure}
Figure \ref{F:elem} below summarizes the known results collated from here, \cite{Egr16} and \cite{Egrsub2}. 

\end{section}

\bibliographystyle{amsplain}
\providecommand{\bysame}{\leavevmode\hbox to3em{\hrulefill}\thinspace}
\providecommand{\MR}{\relax\ifhmode\unskip\space\fi MR }
\providecommand{\MRhref}[2]{%
  \href{http://www.ams.org/mathscinet-getitem?mr=#1}{#2}
}
\providecommand{\href}[2]{#2}
\FloatBarrier

\end{document}